\documentclass[a4paper,12pt]{article}
\usepackage{amsmath, amsfonts, amssymb, amsthm}
\usepackage[english]{babel}

\newcounter{num}[section]

\newenvironment{theorem}
{\refstepcounter{num}%
\bigskip\noindent\nopagebreak[4]{\bf Theorem~\arabic{section}.\arabic{num}. }\it}
\newenvironment{corollary}
{\refstepcounter{num}%
\bigskip\noindent\nopagebreak[4]{\bf Corollary~\arabic{section}.\arabic{num}. }\it}
\newenvironment{lemma}
{\refstepcounter{num}%
\bigskip\noindent\nopagebreak[4]{\bf Lemma~\arabic{section}.\arabic{num}. }\it}
\newenvironment{example}
{\refstepcounter{num}%
\bigskip\noindent\nopagebreak[4]{\bf Example~\arabic{section}.\arabic{num}. }}

\newcommand{\N}{{\mathbb{N}}}

\newcommand{\LL}{{\mathcal{L}}}

\newcommand{\Ss}{{\mathbf{S}}}
\newcommand{\V}{{\mathrm{V}}}

\newcommand{\pr}{{\prime}}

\newcommand{\A}{{\mathcal{A}}}
\newcommand{\Bcal}{{\mathcal{B}}}
\newcommand{\Ecal}{{\mathcal{E}}}
\newcommand{\Gcal}{{\mathcal{G}}}
\newcommand{\Mcal}{{\mathcal{M}}}
\newcommand{\Pcal}{{\mathcal{P}}}

\renewcommand{\a}{{\mathbf{a}}}

\renewcommand{\c}{{\mathbf{c}}}

\newcommand{\Kbf}{{\mathbf{K}}}
\newcommand{\Nbf}{{\mathbf{N}}}

\newcommand{\PG}{{\Pi\Gcal}}

\begin{document}

\title{On equations over direct powers of algebraic structures}
\author{Artem N. Shevlyakov}
\maketitle
\abstract{We study systems of equations over graphs, posets and matroids. We give the criteria, when a direct power of such algebraic structures is equationally Noetherian. Moreover we prove that any direct power of a finite algebraic structure is weakly equationally Noetherian.}

\section{Introduction}

Let $\Kbf$ be an arbitrary class of mathematical objects. One of the main problem of mathematics is to describe ``simple'' and ``hard'' objects in $\Kbf$. One can do it in different ways using various technic of algebra, geometry, calculus etc. In the current paper we make an attempt to classify  ``simple'' and ``hard'' algebraic structures by universal algebraic geometry (UAG). 

Following~\cite{DMR1}, UAG is a discipline of model theory, and it deals with equations over arbitrary algebraic structures. There are many notions of UAG which allow us to separate algebraic structures with ``simple'' and ``hard'' equational properties. The main feature here is the equationally Noetherian property. Recall that an algebraic structure $\A$ is equationally Noetherian if any system of equations $\Ss$ is equivalent over $\A$ to a finite subsystem. Roughly speaking, if an algebraic structure $\A$ is equationally Noetherian, then its equational properties are said to be ``simple''. Otherwise, we assume that $\A$ has a complicated equational theory. 

Indeed, the Noetherian property is a central notion of UAG, and papers~\cite{DMR1,DMR2,DMR3} contain the series of results which establish nice properties of equationally Noetherian algebraic structures. However, for {\it finite} algebraic structures the Noetherian property gives the trivial partition into ``simple'' and ``hard'' classes, since all finite algebraic structures are equationally Noetherian.

Thus, we have to propose an alternative approach in the division of finite algebraic structures into the classes with ``simple'' and ``hard'' equational properties. Our approach satisfies the following:
\begin{enumerate}
\item we deal with lattices of algebraic sets over a given algebraic structures (a set $Y$ is algebraic over an algebraic structure $\A$ if $Y$ is the solution set of an appropriate system of equations);
\item we use the common operations of UAG (direct products, substructures, ultra-products etc.); 
\item the partition into ``simple'' and ``hard'' algebraic structures is implemented by a list of first-order formulas $\Phi$ such that 
\[
\A\mbox{ is ``simple'' }\Leftrightarrow \A\mbox{ satisfies }\Phi.
\]
In other words, the ``simple'' class is axiomatizable by the formulas $\Phi$. 
\end{enumerate}

Namely, we offer to consider infinite direct powers $\Pi \A$ of an algebraic structure $\A$ and study of Diophantine equations over $\Pi\A$ instead of Diophantine equations over $\A$ (an equation $E(X)$ is said to be Diophantine over an algebraic structure $\Bcal$ if $E(X)$ may contain occurrences of any element from $\Bcal$). The decision rule in our approach is the following:
\begin{equation}
\A\mbox{ is ``simple'' }\Leftrightarrow \mbox{all direct powers of }\A\mbox{ are equationally Noetherian };
\label{eq:our_decision_rule}
\end{equation}
otherwise, an algebraic structure $\A$ is said to be ``hard''.   

Some results of the type~(\ref{eq:our_decision_rule}) were obtained in~\cite{shevl_shah}, where we describe all groups, rings and monoids satisfying~(\ref{eq:our_decision_rule}). For example, a group (ring) satisfies~(\ref{eq:our_decision_rule}) iff it is abelian (respectively, with zero multiplication). 

The current paper continues the study of~\cite{shevl_shah}, and in Sections~\ref{sec:graphs}--\ref{sec:matroids} we consider equations over the important classes of relational algebraic structures: graphs, partial orders and matroids. For each of these classes we describe algebraic structures that satisfies~(\ref{eq:our_decision_rule}). 

However, the most complicated and nontrivial part of our paper is Section~\ref{sec:general_results}. It contains the series of general results that hold for any direct power of any finite algebraic structure $\A$. In particular, we prove that any infinite system of equations $\Ss$ over $\Pi\A$ is equivalent to a finite system $\Ss^\pr$ (here we do not claim $\Ss^\pr\subseteq\Ss$). Thus, we prove that any direct power of a finite algebraic structure is weakly equationally Noetherian (see the definition in Section~\ref{sec:basics}).

\section{Basic definitions}
\label{sec:basics}

Following~\cite{DMR1,DMR2,DMR3}, we give the main definitions of universal algebraic geometry.

Let $\LL$ be a language and $\A$ be an algebraic structure of the language $\LL$ ($\LL$-structure). In the current paper we consider languages of the following types: $\LL_g=\{E^{(2)}\}$ (graph language), 
$\LL_p=\{\leq^{(2)}\}$ (partial order language),$\LL_m=\{P_1^{(1)},P_2^{(2)},\ldots\}$ (matroid language). An {\it equation over  $\LL$ ($\LL$-equation)} is an atomic formula over $\LL$.
The examples of equations in various languages are the following: $E(x,y)$, $E(x,x)$, $x=y$ (language $\LL_g$); $x\leq y$, $x\leq x$, $x=y$ (language $\LL_p$) $P_1(x)$, $P_2(x,y)$ $x=y$ (language $\LL_m$).

A {\it system of $\LL$-equations} ({\it $\LL$-system} for shortness) is an arbitrary set of $\LL$-equations. Notice that we consider only systems in a finite set of variables $X=\{x_1,x_2,\ldots,x_n\}$. The set of all solutions of $\Ss$ in an $\LL$-structure $\A$ is denoted by $\V_\A(\Ss)\subseteq \A^n$. A set $Y\subseteq \A^n$ is said to be an {\it algebraic over $\A$} if there exists an $\LL$-system $\Ss$ with $Y=\V_\A(\Ss)$. If the solution set of an $\LL$-system $\Ss$ is empty, $\Ss$ is said to be {\it inconsistent}. Two $\LL$-systems $\Ss_1,\Ss_2$ are called {\it equivalent over an $\LL$-structure $\A$} if $\V_\A(\Ss_1)=\V_\A(\Ss_2)$. This equivalence relation is denoted by $\Ss_1\sim\Ss_2$.

An $\LL$-structure $\A$ is {\it $\LL$-equationally Noetherian} if any infinite $\LL$-system $\Ss$ is equivalent over $\A$ to a finite subsystem $\Ss^\pr\subseteq \Ss$. The class of equationally Noetherian $\LL$-structures is denoted by $\Nbf$.

In~\cite{DMR3} it was introduced generalizations of the Noetherian property.  An $\LL$-structure $\A$ is {\it weakly $\LL$-equationally Noetherian} if any infinite $\LL$-system $\Ss$ is equivalent over $\A$ to a finite system $\Ss^\pr$ (here we do not claim $\Ss^\pr\subseteq\Ss$). The class of weakly equationally Noetherian $\LL$-structures is denoted by $\Nbf^\pr$. Obviously, $\Nbf\subseteq \Nbf^\pr$.

%An $\LL$-structure $\A$ is {\it $\LL$-equationally Noetherian for consistent systems} if any infinite consistent $\LL$-system $\Ss$ is equivalent over $\A$ to a finite subsystem $\Ss^\pr\subseteq\Ss$. The class of such $\LL$-structures is denoted by $\Nbf_c$. One can easily prove that
%\begin{equation}
%\Nbf\subseteq\Nbf_c\subseteq\Nbf^\pr
%\label{eq:N-classes_inclusions}
%\end{equation}
%if the empty set is defined by a finite $\LL$-system in each $\LL$-structure of the given classes.

Let $\A$ be an $\LL$-structure. By $\LL(\A)$ we denote the language $\LL\cup\{a\mid a\in \A\}$ extended by new constant symbols which correspond to elements of $\A$. The language extension allows us to use constants in equations. The examples of equations in extended languages are the following (below $\Gcal$, $\Mcal$ are graph and matroid respectively): $E(x,a)$ (language $\LL_g(\Gcal)$ and $a\in \Gcal$); $P_2(a,x)$, $P_3(x,b,c)$, $P_4(a,x,y,b)$ (language $\LL_m(\Mcal)$ and $a,b,c\in \Mcal$). Obviously, the class of $\LL(\A)$-equations is wider than the class of $\LL$-equations, so an $\LL$-equationally Noetherian $\LL$-algebra may lose this property in the language $\LL(\A)$. %5Recall that the inclusions~(\ref{eq:N-classes_inclusions}) hold for the language $\LL(\A)$, since the empty set is defined by the equation $a_1=a_2$ for distinct $a_1,a_2\in\A$. 

Let $\A$ be an $\LL$-structure. An element of a {direct power} $\Pi \A=\prod_{i\in I}\A$ is denoted by  a sequence in square brackets $[a_i\mid i\in I]$. Functions and relations over $\Pi\A$ have the coordinate-wise definition. For example, any relation $R^{m}\in\LL$ is defined on $\Pi\A$ as follows:
\[
R([a^{(1)}_i\mid i\in I],[a^{(2)}_i\mid i\in I],\ldots,[a^{(m)}_i\mid i\in I])\Leftrightarrow 
R(a^{(1)}_i,a^{(2)}_i,\ldots,a^{(m)}_i) \mbox{ for each }i\in I.
\]
The map $\pi_k\colon\Pi\A\to \A$ is called a {\it projection onto the $i$-th coordinate} if $\pi_k([a_i\mid i\in I])=a_k$.

Let $E(X)$ be an $\LL(\Pi\A)$-equation over the direct power $\Pi\A$. We may rewrite $E(X)$ in the form $E(X,\overrightarrow{\mathbf{C}})$, where $\overrightarrow{\mathbf{C}}$ is an array of constants occurring in the equation $E(X)$. One can introduce the {\it projection of an equation} onto the $i$-th coordinate as follows:
\[
\pi_i(E(X))=\pi_i(E(X,\overrightarrow{\mathbf{C}}))=E(X,\pi_i(\overrightarrow{\mathbf{C}})),
\] 
where $\pi_i(\overrightarrow{\mathbf{C}}))$ is an array of the $i$-th coordinates of the elements from $\overrightarrow{\mathbf{C}}$. For example, the $\LL_g(\Pi\Gcal)$-equation $E(x,[a_1,a_2,a_3,\ldots])$ has the following projections
\begin{eqnarray*}
E(x,a_1),\\
E(x,a_2),\\
E(x,a_3),\\
\ldots
\end{eqnarray*}
Similarly, the matroid equation $P_4(x,[a_1,a_2,a_3,\ldots],y,[b_1,b_2,b_3,\ldots])$ has the projections 
\begin{eqnarray*}
P_4(x,a_1,y,b_1),\\
P_4(x,a_2,y,b_2),\\
P_4(x,a_3,y,b_3),\\
\ldots
\end{eqnarray*}

Let us take an $\LL(\Pi\A)$-system $\Ss=\{E_j(X)\mid j\in J\}$. The $i$-th projection of $\Ss$ is the $\LL(\A)$-system $\pi_i(\Ss)=\{\pi_i(E_j(X))\mid j\in J\}$. The projections of an $\LL(\Pi\A)$-system $\Ss$ allow to describe the solution set of $\Ss$ by
\begin{equation}
\V_{\Pi\A}(\Ss)=\{[P_i\mid i\in I]\mid P_i\in \V_{\A}(\pi_i(\Ss))\}.
\label{eq:solution_set_via_projections}
\end{equation}
In particular, if one of the projections $\pi_i(\Ss)$ is inconsistent, so is $\Ss$.

The following statement immediately follows form~(\ref{eq:solution_set_via_projections}).

\begin{lemma}
Let $\Ss=\{E_j(X)\mid j\in J\}$ be an $\LL(\Pi\A)$-system over $\Pi\A$. The system $\Ss$ is consistent iff so are all projections $\pi_i(\Ss)$. Moreover, if $\A$ is $\LL$-equationally Noetherian, then an inconsistent $\LL(\Pi\A)$-system $\Ss$ is equivalent to a finite subsystem.
\label{l:inconsistent_systems}
\end{lemma}
\begin{proof}
The first assertion directly follows from~(\ref{eq:solution_set_via_projections}). Suppose $\A$ is $\LL$-equationally Noetherian, and $\pi_i(\Ss)$ is inconsistent. Hence, $\pi_i(\Ss)$ is equivalent to its finite inconsistent subsystem $\{\pi_i(E_j(X))\mid j\in J^\pr\}$, $|J^\pr|<\infty$, and the subsystem $\Ss^\pr=\{E_j(X)\mid j\in J^\pr\}$ of $\Ss$ is also inconsistent.
\end{proof}

\section{Graphs}
\label{sec:graphs}

Recall that a {\it graph} is an algebraic structure of the language $\LL_g=\{E^{(2)}\}$ satisfying the following axioms:
\begin{eqnarray*}
\forall x\; \neg E(x,x)\mbox{ (no loops)},\\
\forall x\forall y\; E(x,y)\to E(y,x)\mbox{ (symmetry)}.
\end{eqnarray*} 

\begin{theorem}
A graph $\Pi\Gcal=\prod_{i\in I}\Gcal$ is $\LL_g(\PG)$-equationally Noetherian iff $\Gcal$ satisfies the quasi-identity
\begin{equation}
\forall x_1\forall x_2\forall x_3 \forall x_3\:\left( E(x_1,x_2)\wedge E(x_2,x_3)\wedge E(x_3,x_4)\to E(x_4,x_1) \right).
\label{eq:quasi-identity_N}
\end{equation}
\label{th:N_criterion}
\end{theorem}
\begin{proof}
Let us prove the ``if'' part of the statement.

Let $\Ss$ be an $\LL_g(\Pi\Gamma)$-system over $\PG$ in variables $X=\{x_1,\ldots,x_n\}$. One can rewrite $\Ss$ as a finite union of systems
\[
\Ss=\bigcup_{j=1}^n\Ss_{j}\bigcup \Ss_{0},
\]
where $\Ss_j=\{E(x_j,\c_k)\mid k\in K_j\}$ and $\Ss_{0}$ is the system of equations of the following types: $E(x_i,x_j)$, $x_i=x_j$, $x_i=\c_j$. Obviously, the system $\Ss_{0}$ is equivalent to a finite subsystem. Hence, it is sufficient to prove that each system $\Ss_j$ in one variable $x_j$ is equivalent to its finite subsystem.

Let us write the coordinate-wise versions of the system $\Ss_j$:
\[
\pi_i(\Ss_j)=\{E(x_j,\pi_i(\c_{k}))\mid k\in K_j\},\;  i\in I,
\] 
where $\pi_i(\c_k)$ is the $i$-th coordinate of an element $\c_k$.

If for each $i$ the equations $\{E(x_j,\pi_i(\c_{k}))\mid k\in K_j\}$ have the same solution sets, then $\Ss_j$ is equivalent to a single equation $E(x_j,\c_k)\in\Ss_j$ for arbitrary $k\in K_j$. Otherwise, there exists an index $i$ such that
\begin{equation}
Y_1=\V_\Gcal(E(x_j,\pi_i(\c_{k_1})))\neq\V_\Gcal(E(x_j,\pi_i(\c_{k_2})))=Y_2
\label{eq:313231}
\end{equation}
for some $k_1,k_2\in K_j$.

If $Y_1\cap Y_2=\emptyset$, then $\Ss_j$ is inconsistent and it is obviously equivalent to the subsystem $\{E(x_j,\c_{k_1}),E(x_j,\c_{k_2})\}$. Thus, we may assume $Y_1\nsubseteq Y_2$ and one can take elements  $b_1,b_2\in\Gcal$ such that $b_1\in Y_1\setminus Y_2$, $b_2\in Y_1\cap Y_2$, i.e. $E(b_1,\pi_i(\c_{k_1}))$, $E(b_2,\pi_i(\c_{k_1}))$ and $E(b_2,\pi_i(\c_{k_2}))$. 

Since the quasi-identity~(\ref{eq:quasi-identity_N}) is true in $\PG$, we have $E(b_1,\pi_i(\c_{k_2})$ that contradicts the choice of the element $b_1$.

\medskip

Let us prove the ``only if'' part of the statement. Assume the quasi-identity~(\ref{eq:quasi-identity_N}) does not hold in $\Gcal$, i.e. there exists elements $a_1,a_2,a_3,a_4$ with $E(a_1,a_2)$, $E(a_2,a_3)$, $E(a_3,a_4)$, $\neg E(a_4,a_1)$. Consider the $\LL_g(\PG)$-system $\Ss$ of the following equations:
\begin{eqnarray*}
E(x,[a_2,a_2,a_2\ldots]),\\
E(x,[a_4,a_2,a_2\ldots]),\\
E(x,[a_4,a_4,a_2\ldots]),\\
\ldots
\end{eqnarray*}  
Let $\Ss_n$ be the subsystem of $\Ss$ formed by the first $n$ equations of $\Ss$. 

The point $\a=[\underbrace{a_3,a_3,\ldots,a_3}_{\mbox{$n-1$ times}},a_1,a_1,\ldots]$ satisfies $\Ss_n$ but $\a$ does not satisfy the $(n+1)$-th equation of $\Ss$. Thus, $\Ss_n$ is not equivalent to $\Ss$ for any $n$, and $\PG$ is not $\LL_g(\PG)$-equationally Noetherian.
\end{proof}

\begin{corollary}
If a graph $\Gcal$ contains a triangle (i.e. there exist vertices $x_1,x_2,x_3\in\Gcal$ with $E(x_1,x_2)$, $E(x_2,x_3)$, $E(x_3,x_1)$) then $\PG$ is not $\LL_g(\PG)$-equationally Noetherian. 
\label{cor:triangle_free}
\end{corollary}
\begin{proof}
Obviously, the condition of Theorem~\ref{th:N_criterion} fails for such graphs, since there are not loops in $\Gcal$.
\end{proof}

Let $\Kbf=\{\Gcal\mid \PG\in\Nbf\}$ be the set of all graphs with equationally Noetherian direct powers. Theorem~\ref{th:N_criterion} gives that the class $\Kbf$ is axiomatizable. The class $\Kbf$ may be also described by forbidden graphs and distance functions. 

\begin{corollary}
A graph $\PG$ is $\LL_g(\PG)$-equationally Noetherian iff $\Gcal$ is triangular-free and the distance between any pair of vertices $x,y$ is either $\infty$ (if $x,y$ belong to different connected components) or less than $4$. 
\end{corollary}
\begin{proof}
First, we prove the ``only if'' part of the statement. By Corollary~\ref{cor:triangle_free} $\Gcal$ is triangular-free. Let us take two vertices $x,y$ with the distance $4\leq d(x,y)=d<\infty$ and the shortest path $x=x_1,x_2,\ldots,x_d=y$ between $x$ and $y$. However, the quasi-identity~(\ref{eq:quasi-identity_N}) provides that $E(x_1,x_4)$, and the minimal path between $x,y$ has the length less than $d$, a contradiction.

Let us prove the ``if'' part of the statement and take arbitrary $x_1,x_2,x_3,x_4\in\Gcal$ such that $E(x_1,x_2)$, $E(x_2,x_3)$, $E(x_3,x_4)$. Since the distance between vertices of the same connected component is less or equal than $3$, then there exists an edge between the vertices $x_i$. If there exists one of the edges $E(x_1,x_3)$, $E(x_2,x_4)$ then $\Gcal$ contains a triangle. Thus, $\Gcal$ has the edge $E(x_1,x_4)$ and the quasi-identity~(\ref{eq:quasi-identity_N}) holds in $\Gcal$.
\end{proof}

Let us give the explicit examples of graphs $\Gcal\in\Kbf$.

One can directly prove that the disjoint union $\Gcal=\Gcal_1\sqcup\Gcal_2$ has an equationally Noetherian direct power $\PG$ if both graphs satisfy the quasi-identity~(\ref{eq:quasi-identity_N}). Thus, there arises a question: is there a {\it connected} graph $\Gcal$ with $n$ vertices such that any direct power $\PG$ is $\LL_g(\PG)$-equationally Noetherian?

The answer is positive. Let us define the following graph $\Gcal$ with the vertex set $\{x_0,x_1,\ldots,x_n,x_{n+1}\}$ and edges $\{E(x_0,x_i),E(x_i,x_{n+1})\mid 1\leq i\leq n\}$. The direct check gives that $\Gcal$ satisfies~(\ref{eq:quasi-identity_N}), contains $n+1$ vertices and $\Gcal$ is connected.

\section{Partial orders}
\label{sec:posets}

A {\it partial order} $\Pcal$ is an algebraic structure of the language $\LL_p=\{\leq^{(2)}\}$ such that $\Pcal$ satisfies the following axioms
\begin{eqnarray*}
\forall x\; (x\leq x),\\
\forall x\forall y\; (x\leq y)\wedge(y\leq x)\to(x=y),\\
\forall x\forall y\; (x\leq y)\wedge(y\leq z)\to(x\leq z).
\end{eqnarray*}

A partial order $\Pcal$ is said to be {\it non-trivial} if there exists a pair $a,b\in\Pcal$ such that $a<b$ (i.e. $a\leq b$ and $a\neq b$).

%It follows that an infinite direct power of a non-trivial partial order $\Pcal$ is not equationally Noetherian.   

\begin{theorem}
Let $\mathcal{P}$ be a non-trivial partial order, and $\Pi\mathcal{P}$ be an infinite direct power of $\mathcal{P}$. Then $\Pi\mathcal{P}$ is not $\LL_p(\Pi\Pcal)$-equationally Noetherian.
\end{theorem}
\begin{proof}
Since $\Pcal$ is non-trivial, there exists $a,b\in\Pcal$ with $a<b$. It is sufficient to show that an infinite direct power $\Pi \Ecal\subseteq \Pi\Pcal$ of the partial order $\Ecal=\{a,b\}$ is not $\LL_p(\Pi \Ecal)$-equationally Noetherian.

Indeed, one should consider the following infinite $\LL_p(\Pi \Ecal)$-system $\Ss$: 
\begin{eqnarray*}
x\leq[b,b,b,\ldots,]\\
x\leq[a,b,b,\ldots,]\\
x\leq[a,a,b,\ldots,]\\
\ldots
\end{eqnarray*}
Obviously, the unique solution of $\Ss$ is $[a,a,a,\ldots,]$. However the solution set of any finite subsystem of $\Ss$ contain a point $[\underbrace{a,a,a,\ldots,a}_{\mbox{$n$ times}},b,b,b\ldots,]$ for sufficiently large $n$. Thus, $\Ss$ is not equivalent to any finite subsystem.
\end{proof}

\section{Matroids}
\label{sec:matroids}

One can consider a {\it matroid} $\Mcal$ as an algebraic structure of an infinite language $\LL_m=\{P_1^{(1)},P_2^{(2)},P_3^{(3)},\ldots\}$, where each predicate symbol $P_n$ have the following interpretation:
\[
P_n(x_1,\ldots,x_n)\Leftrightarrow \mbox{ the set $\{x_i\}$ is independent in $\Mcal$}.
\]   
Moreover, any matroid satisfies the following axioms:
\begin{eqnarray*}
\forall x_1\ldots\forall x_n\; \left(\bigvee_{i\neq j}(x_i=x_j)\to\neg P_n(x_1,\ldots,x_n) \right)\\
\forall x_1\ldots\forall x_n\; \left(P_n(x_1,\ldots,x_n)\to \bigwedge_{i=1}^n P_{n-1}(x_1,\ldots,x_{i-1},x_{i+1},\ldots,x_n)\right)\; (n>1),\\
\forall x_1\ldots\forall x_n\; \left(P_n(x_1,\ldots,x_n)\wedge P_{n+1}(y_1,\ldots,y_{n+1})\to \bigvee_{i=1}^{n+1} P_{n+1}(x_1,\ldots,x_n,y_i)\right).
\end{eqnarray*}

{\it Notice that a direct power $\Pi \Mcal$ of a matroid $\Mcal$ is not necessarily a monoid itself}. However, here we study direct powers of matroids, since the algebraic geometry over $\Pi\Mcal$ may clarify algebraic and geometric properties of the original matroid $\Mcal$. 

\begin{lemma}
Let $\Mcal$ be a matroid with $P_3(a,b,c)$ for some $a,b,c\in\Mcal$. Then any infinite direct power $\Pi\Mcal$ is not $\LL_m(\Pi\Mcal)$-equationally Noetherian.
\label{l:matroids_with_triples}
\end{lemma}
\begin{proof}
Let us consider a system $\Ss$ of $\LL_m(\Pi \Mcal)$-equations
\begin{eqnarray*}
P_2(x,[a,a,a,\ldots]),\\
P_2(x,[b,a,a,\ldots]),\\
P_2(x,[b,b,a,\ldots]),\\
\ldots
\end{eqnarray*}
Denote by $\Ss_n$ the first $n$ equations of $\Ss$. Clearly, $\Ss_n$ is satisfied by the point
\[
[\underbrace{c,c,\ldots,c}_{\mbox{ $n$ times}},b,b,\ldots,].
\]
However this point does not belong to the solution set of $\Ss$, since the predicate 
\[
P_2([\underbrace{c,c,\ldots,c}_{\mbox{ $n$ times}},b,b,\ldots,],[\underbrace{b,b,\ldots,b}_{\mbox{$n+1$ times}},a,a,\ldots,])
\]
is not true for the $(n+1)$-th coordinate. 
\end{proof}

According to Lemma~\ref{l:matroids_with_triples}, any matroid $\Mcal$ with $\Pi\Mcal\in\Nbf$ may be represented by a graph $\Gcal(\Mcal)$ such that
\begin{enumerate}
\item the vertex set of $\Gcal$ coincides with the set $\Mcal$;
\item $P_2(a,b)\Leftrightarrow E(a,b)$. 
\end{enumerate}

Hence, such matroids may be classified by the analogue of Theorem~\ref{th:N_criterion}. 

\begin{theorem}
A direct power $\Pi\Mcal$ of a matroid $\Mcal$ is $\LL_m(\Mcal)$-equationally Noetherian iff $\Mcal$ satisfies the following axioms
\begin{eqnarray*}
\forall x\forall y\forall z\; \neg P_3(x,y,z),\\
\forall x_1\forall x_2\forall x_3\forall x_4\:\left(P_2(x_1,x_2)\wedge P_2(x_2,x_3)\wedge P_2(x_3,x_4)\to P_2(x_4,x_1) \right).
\end{eqnarray*}
\end{theorem}
\begin{proof}
The proof immediately follows from Lemma~\ref{l:matroids_with_triples}, Theorem~\ref{th:N_criterion} and the correspondence $\Mcal\leftrightarrow\Gcal(\Mcal)$. 
\end{proof}

\section{Direct powers of finite structures}
\label{sec:general_results}

Let us prove a general fact about direct powers of arbitrary finite algebraic structures. The proof of the following theorem is complicated enough, so its main steps are explained in Example~\ref{ex:ex}.

\begin{theorem}
Let $\A$ be a finite $\LL$-structure. Then any direct power $\Pi \A=\Pi_{i\in I}\A$ is weakly $\LL(\Pi \A)$-equationally Noetherian.
\label{th:direct_product_finite}
\end{theorem}
\begin{proof}

Let $\Ss=\{E_j(X,\overrightarrow{\mathbf{C_j}})\mid j\in J\}$ be an infinite $\LL(\Pi\A)$-system over $\Pi\A$, and $\pi_i(\Ss)=\{E_j(X,\pi_i(\overrightarrow{\mathbf{C_j}}))\mid j\in J\}$ ($i\in I$) be the projections of $\Ss$ onto all coordinates of $\Pi\A$. Notice that any system $\pi_i(\Ss)$ is a system of $\LL(\A)$-equations over $\A$.

Since $\A$ is finite,  then  there exists a finite number of equations $M=\{E_j(X,\pi_i(\overrightarrow{\mathbf{C_j}}))\mid (i,j)\in K\}$ ($|K|<\infty$) such that any $E_j(X,\pi_i(\overrightarrow{\mathbf{C_j}}))\in \bigcup_{i\in I}\pi_i(\Ss)$ is equivalent over $\A$ to an appropriate equation from $M$. Hence, each $\pi_i(\Ss)$ is equivalent to a subsystem $\Ss_i^\pr\subseteq M$ over $\A$. The idea of the further proof is the following: we try to wrap all systems $\Ss_i^\pr$ into a finite number of equations $\Ss^\pr$ over $\Pi\A$. 

Let us define an $\LL(\Pi \A)$-system $\Ss^\pr$ by the following procedure.

{\bf Step 0.} Put 
\[
\Ss_0=\bigcup_{(i,j)\in K} E_j(X,\overrightarrow{\mathbf{C_j}})\subseteq \Ss
\]
($|\Ss_0|=|K|$) and $\Ss^\pr:=\Ss_0$. The main property of $\Ss_0$ is the following: each equation from $M$ occurs in some projection of equations from $\Ss_0$. Let us arbitrarily enumerate equations in the set $M$, i.e. each equations from $M$ has the number $s\in[1,|K|]$.

{\bf Step $s$} ($1\leq s\leq |K|$). Let us take the $s$-th equation $E_j(X,\pi_i(\overrightarrow{\mathbf{C_j}}))$ from $M$ and define the following sets of indexes $I_0=\{l\in I\mid E_j(X,\pi_i(\overrightarrow{\mathbf{C_j}}))\in\Ss_l^\pr\}$, $I_1=I\setminus I_0$. In other words, $I_0$ is the set of all indexes $l$ such that the given equation from $M$ occurs in the system $\Ss^\pr_l$. Define a set $M_{s}=\{D_l(X)\mid l\in I\}$ of $\LL(\A)$-equations as follows:
\[
D_l(X)=\begin{cases}
E_j(X,\pi_i(\overrightarrow{\mathbf{C_j}}))\mbox{ if }l\in I_0,\\
E_j(X,\pi_l(\overrightarrow{\mathbf{C_j}}))\mbox{ if }l\in I_1
\end{cases}
\]
The sense of the set $M_{s}$ is the following. If the system $\Ss^\pr_l$ contains $E_j(X,\pi_i(\overrightarrow{\mathbf{C_j}}))\in M$ we take this equation as the $k$-th projection in $M_{s}$. Otherwise, the $l$-th projection in $M_{s}$ is taken from the equation $E_j(X,\overrightarrow{\mathbf{C_j}})\in\Ss_0$.

The $\LL(A)$-equations $M_{s}$ may be wrapped into the $\LL(\Pi\A)$-equation $D_{s}(X,\overrightarrow{\mathbf{D_{s}}})$, where 
\[
\pi_l(\overrightarrow{\mathbf{D_{s}}})=\begin{cases}
	\pi_i(\overrightarrow{\mathbf{C_j}})\mbox{ if }l\in I_0,\\
\pi_l(\overrightarrow{\mathbf{C_j}})\mbox{ if }l\in I_1
	
	\end{cases}
\]

We put $\Ss^\pr:=\Ss^\pr\cup D_{s}(X,\overrightarrow{\mathbf{D_{s}}})$ and go to the following step $(s+1)$.

\medskip

By the definition of the system $\Ss^\pr$, the $i$-th projection $\pi_i(\Ss^\pr)$ contain all equations from $\Ss_i^\pr\sim \pi_i(\Ss)$. Hence, $\pi_i(\Ss^\pr)\sim \pi_i(\Ss)$ over $\A$, and finally $\Ss^\pr\sim \Ss$ over $\Pi\A$.

%Now we should prove that $\Ss$ is equivalent to $\Ss^\pr$ over $\Pi\A$. First, we observe that {\it any $\LL$-equation $\pi_i(E(X))$ (where $E(X)\in\Ss^\pr$) either came from $\Ss_i^\pr$ or is equal to the $i$-th projection of some equation from $\Ss_0$}.
%
%Let us consider a point $\mathbf{P}=(P_1,\ldots,P_n)\in(\Pi\A)^n$, $P_k=[p_{i}^{(k)}\mid i\in I]$. By $\pi_i(\mathbf{P})$ we denote the point $(p_i^{(1)},\ldots,p_i^{(k)})\in\A^n$.  
%
%Assume $\mathbf{P}\notin \V_{\Pi\A}(\Ss^\pr)$. Hence, there exists an $\LL$-equation $\pi_i(E(X))$ ($E(X)\in\Ss^\pr$) with $\pi_i(\mathbf{P})\notin\V_{\A}(\pi_i(E(X)))$. Thus, $\mathbf{P}$ does not satisfy the system $\Ss$.
%
%Suppose now $\mathbf{P}\notin\V_{\Pi\A}(\Ss)$. There exists an $\LL(\Pi\A)$-equation $E(X)\in\Ss$ and a projection $\pi_k(E(X))$ such that $\pi_k(\mathbf{P})\notin\pi_k(E(X))$. Hence, the point $\pi_k(\mathbf{P})$ does not satisfy the $\LL(\A)$-system $\Ss_k^\pr$. Since the $\LL(\A)$-system $\Ss^\pr_k$ consists of equations of the set $M$, there exists an $\LL(\A)$-equation $E_j(X,\pi_i(\overrightarrow{\mathbf{C_j}}))\in M$ with $\pi_k(\mathbf{P})\notin E_j(X,\pi_i(\overrightarrow{\mathbf{C_j}}))$. By the definition of the set $M_{ij}$, the equation $E_j(X,\pi_i(\overrightarrow{\mathbf{C_j}}))$ is the $k$-th projection of the $\LL(\Pi\A)$-equation $D_{ij}(X,\overrightarrow{\mathbf{D_{ij}}})$. Hence, the point $\mathbf{P}$ does not satisfy the $\LL(\Pi\A)$-equation $D_{ij}(X,\overrightarrow{\mathbf{D_{ij}}})$. Thus, $\mathbf{P}\notin\Ss^\pr$.

\end{proof}

The following example explains the technique and denotations from Theorem~\ref{th:direct_product_finite}.

\begin{example}
\label{ex:ex}
Let $\Gcal$ be the graph with vertices $\{a,b,c\}$ and edges $E(a,b)$, $E(b,c)$, $E(c,a)$ (i.e. $\Gcal$ is a complete graph). Let us consider an infinite $\LL(\Pi\Gcal)$-system $\Ss$ of equations:
\begin{eqnarray*}
E(x,[a,a,a,a,a,a,\ldots]),\\
E(x,[b,a,a,a,a,a,\ldots]),\\
E(x,[b,c,a,a,a,a,\ldots]),\\
E(x,[b,c,b,a,a,a,\ldots]),\\
E(x,[b,c,b,c,a,a,\ldots]),\\
\ldots
\end{eqnarray*}
The projections $\pi_i(S)$ are the following (we omit in the projections equations which occur earlier):
\begin{eqnarray*}
\pi_1(\Ss)=\{E(x,a),E(x,b)\},\\
\pi_2(\Ss)=\{E(x,a),E(x,c)\},\\
\pi_3(\Ss)=\{E(x,a),E(x,b)\},\\
\pi_4(\Ss)=\{E(x,a),E(x,c)\},\\
\ldots
\end{eqnarray*}

The set $M$ consists of the equations $E(x,a)$, $E(x,b)$, $E(x,c)$ (any equation from $\bigcup_{i=1}^n\pi_i(\Ss)$  is equivalent to one of the given equations). Since the third equation of $\Ss$ contain all equations from $M$ as projections, we may put $\Ss_0=\{E(x,[b,c,a,a,a,a,\ldots])\}$ (the set $K$ here is $\{(1,3),(2,3),(3,3)\}$). For the projections $\pi_i(\Ss)$ we have
\begin{eqnarray*}
\pi_{2k+1}(\Ss)&\sim&\{E(x,a),E(x,b)\}=\Ss_{2k+1}^\pr,\\
\pi_{2k}(\Ss)&\sim&\{E(x,a),E(x,c)\}=\Ss_{2k}^\pr.
\end{eqnarray*}

Now we construct the final system $\Ss^\pr$ with $|\Ss_0|+|M|=4$ equations. First, we put $\Ss^\pr=\Ss_0$ and make the following three steps.

\begin{enumerate}
\item We take $E(x,a)\in M$. Since this equation occurs in any system $\Ss_i^\pr$ ($I_0=\N$, $I_1=\emptyset$), we add to $\Ss^\pr$ the equation $E(x,[a,a,a,a,a,\ldots])$.
\item Take $E(x,b)\in M$. Since $E(x,b)$ occurs in the systems $\Ss_i$ with odd $i$ ($I_0=\{1,3,\ldots\}$, $I_1=\{2,4,\ldots\}$), we should add to $\Ss^\pr$ an equation of the form $E(x,[b,\ast,b,\ast,b,\ast,\ldots])$. The elements for even positions are taken from the equation from $\Ss_0$, and we obtain the equation $E(x,[b,c,b,a,b,a,\ldots])$. The last equation is added to $\Ss^\pr$.
\item For the equation $E(x,c)\in M$ we make dual operations. Since $E(x,c)$ occurs in the systems $\Ss_i$ with even $i$ ($I_0=\{2,4,\ldots\}$, $I_1=\{1,3,\ldots\}$) then we should add to $\Ss^\pr$ an equation of the form $E(x,[\ast,c,\ast,c,\ast,c,\ldots])$. The elements for odd positions are taken from the equation from $\Ss_0$, and we obtain the equation $E(x,[b,c,a,c,a,c,\ldots])$. Also we add the last equation to $\Ss^\pr$.
\end{enumerate} 
Thus, the final system $\Ss^\pr$ consists of the following equations
\begin{eqnarray*}
E(x,[b,c,a,a,a,a,\ldots]),\\
E(x,[a,a,a,a,a,a,\ldots]),\\
E(x,[b,c,b,a,b,a,\ldots]),\\
E(x,[b,c,a,c,a,c,\ldots]).
\end{eqnarray*}
It is easy to see that all projections $\pi_i(\Ss^\pr)$ are equivalent over $\Gcal$ to the systems $\Ss_i^\pr$. Thus, $\Ss^\pr$ is equivalent to $\Ss$.
\end{example}

\medskip

The ideas of Theorem~\ref{th:direct_product_finite} allow us to estimate uniformly the minimal number of equations in the finite system $\Ss^\pr$.

\begin{corollary}
Let $\Ss$ be a system of $\LL(\Pi\A)$-equations in $n$ variables over the direct power $\Pi\A$ of a finite $\LL$-structure $\A$, $|\A|=k$. Then $\Ss$ is equivalent to a system $\Ss^\pr$ with at most $2^{k^n+1}$ equations.
\end{corollary}
\begin{proof}
Since we deal with equations in $n$ variables, all algebraic sets over $\A$ are the subsets of the affine space $\A^n$, $|\A^n|=k^n$. Hence, there exists at most $2^{k^n}$ different algebraic sets over $\A$. Since the set $M$ in Theorem~\ref{th:direct_product_finite} consists of pairwise non-equivalent equations, we have $|M|\leq 2^{k^n}$.

The final system $\Ss^\pr$ consists of at most $|M|+|M|=2|M|$ equations ($|\Ss_0|=|M|$, and $|M|$ iterations of the procedure add to $\Ss^\pr$ exactly $|M|$ equations). Thus, we obtain $|\Ss^\pr|\leq 2\cdot 2^{k^n}=2^{k^n+1}$.
\end{proof}

The information of the author:

Artem N. Shevlyakov

Sobolev Institute of Mathematics

644099 Russia, Omsk, Pevtsova st. 13

\medskip

Omsk State Technical University

pr. Mira, 11, 644050

Phone: +7-3812-23-25-51.

e-mail: \texttt{a\_shevl@mail.ru}

\end{document}